\documentclass[10pt]{elsarticle}
\usepackage{amsmath,amssymb,amsfonts,amsthm}
\usepackage{graphicx}
\usepackage[svgnames]{xcolor}
\usepackage{tikz}
\usetikzlibrary{decorations.markings}
\usetikzlibrary{shapes.geometric}

\pgfdeclarelayer{edgelayer}
\pgfdeclarelayer{nodelayer}
\pgfsetlayers{edgelayer,nodelayer,main}

\tikzstyle{Subgraph}=[circle, thick, draw=black, fill=white, scale=2]
\tikzstyle{Small}=[circle,thick, draw=black,fill=black,scale=1]
\tikzstyle{Big}=[circle,thick, draw=black,fill=white,scale=2]
\tikzstyle{new}=[diamond,thick, draw=black,fill=white,scale=1.5]
\tikzstyle{Directed}=[{thick,decoration={
  markings,
  mark=at position .5 with {\arrow{>}}},postaction={decorate}}]

\newtheorem{thm}{Theorem}

\newtheorem{lemma}[thm]{Lemma}

\newtheorem{corollary}[thm]{Corollary}

\theoremstyle{remark}
\newtheorem{rmk}{Remark}[section]

\pgfdeclarelayer{edgelayer}
\pgfdeclarelayer{nodelayer}
\pgfsetlayers{edgelayer,nodelayer,main}

\tikzstyle{Subgraph}=[fill=white, draw=black, shape=circle]
\tikzstyle{smallBlack}=[fill=black, draw=black, shape=circle, minimum size=2mm]
\tikzstyle{grayedOut}=[fill={rgb,255: red,191; green,191; blue,191}, draw=black, shape=circle, minimum size=2mm]

\tikzstyle{Deleted}=[-, dashed]
\tikzstyle{new edge style 1}=[-]

\usepackage{enumitem}

\newlist{proofcases}{enumerate}{7}
\setlist[proofcases]{label*=\arabic*.,leftmargin=*}%

\numberwithin{equation}{section} \numberwithin{thm}{section}
\numberwithin{rmk}{section} \numberwithin{figure}{section}

\newcommand{\diam}{\text{diam}}

\newcommand{\ordiam}{\overrightarrow{\text{diam}}}
\newcommand{\ori}{\overrightarrow}
\newcommand{\oriback}{\overleftarrow}

\begin{document}

\title{Large Girth and Small Oriented Diameter Graphs}


\author[Cochran]{Garner Cochran}
\address{Department of Mathematics and Computer Science\\ Berry College\\2277 Martha Berry Hwy NW\\  Mt Berry GA 30149\\ USA}
\ead{gcochran@berry.edu}


\begin{abstract}
	
	In 2015, Dankelmann and Bau proved that for every bridgeless graph $G$ of order $n$ and minimum degree $\delta$ there is an orientation of diameter at most $11\frac{n}{\delta+1}+9$. In 2016, Surmacs reduced this bound to $7\frac{n}{\delta+1}.$ In this paper, we consider the girth of a graph $g$ and show that for any $\varepsilon>0$ there is a bound of the form $(2g+\varepsilon)\frac{n}{h(\delta,g)}+O(1)$, where $h(\delta,g)$ is a polynomial. Letting $g=3$ and $\varepsilon<1$ gives an inprovement on the result by Surmacs.

\begin{keyword}
		diameter \sep oriented diameter \sep orientation \sep oriented graph \sep distance \sep size \sep girth
	\end{keyword}	
\end{abstract}

\maketitle 

\section{Definitions}

Let $G=(V,E)$ denote a finite simple graph with vertex set $V$ and edge set $E\subseteq \binom{V}{2}$. Given $G=(V,E)$, a subgraph $H$ of $G$, denoted $H\subseteq G$, is a graph $H=(V',E')$ for which $V'\subseteq V$ and $E'\subseteq E\cap \binom{V'}{2}$. By $|G|$ we mean the order of $G$, $|V(G)|$. A digraph $\overrightarrow{G}=(V,A)$ is a graph with a vertex set $V$ and an arc set $A$ where each arc is oriented and the orientation of the arc $a$ with ends $u$ and $v$ is in the direction from $u$ to $v$ will be denoted as $\overrightarrow{uv}$. If a set of arcs $A$ when considered to be unordered is the set $E$, we call $\ori{G}$ an orientation of the graph $G$. A path is defined as $P=(V,E)$, where $V=\{v_0,v_1,\dots,v_n\}$ and $E=\{x_0x_1,x_1x_2,\dots,x_{n-1}x_n\}$. We will denote this path $P=v_0v_1\dots v_n$. Given such a path $P$, a cycle is defined as a graph $G=(V(P),E(P)\cup \{v_0 v_n\}).$ Given an unoriented path $P=v_0v_1\dots v_n$, we denote using $\ori{P}$ the corresponding oriented path from $v_0$ to $v_n$, we will denote using $\oriback{P}$ the oriented path from $v_n$ to $v_0$. Denote the interior of a path $\overline{P}=v_1\dots v_{n-1}.$ Given a graph $G$ and an edge set $E'\subseteq E$, define $G\setminus E'=(V,E\setminus E')$. Given an edge set containing a single edge, $E'=\{e\}$, we may leave off the brackets, i.e. $G\setminus \{e\}=G\setminus e=(V,E\setminus\{e\}).$ We define a forest as a graph containing no cycles. A connected forest is called a tree.

For a set $B\subseteq V(G)$, the induced subgraph of $G$ on the vertex set $B$ is denoted by $G[B]$. That is, $G[B]=\left(B,\binom{B}{2}\cap V\right)$. Given $G$ a simple graph and $v\in V(G)$, the degree of $v$ in $G$ is the number of vertices adjacent to $v$, denoted $\deg(v)=|\{uv \mid u\in V(G), u\neq v, uv\in E(G)\}|.$ The minimum degree of a graph $G$ is $\delta(G)=\min\{\deg(v)\mid v\in V(G)\}$. If the graph $G$ is unambiguous, we let $\delta(G)=\delta$. We define the closed neighborhood of a vertex $v\in V(H)$ in the given subgraph $H$ as, $N_H[v]=\{u \mid u=v \text{ or } uv\in E(H)\}$. The open neighborhood of $v$ in a given subgraph $H$, denoted $N_H(v)$, is defined as $N_H(v)=\{u \mid u\neq v \text{ and } uv\in E(H)\}$. We may also use $N[v]$ and $N(v)$ if the subgraph $H$ is unambiguous. Let $g(G)=g$ be the \textit{girth} of $G$ or the length of the smallest cycle in the graph $G$. 



We define the \textit{distance} from $u$ and $v$ in a graph $G$ or digraph $\ori{G}$ as the minimum number of edges or arcs on a path from $u$ to $v$. We denote this as $\rho_G(u,v)$ or $\rho_{\ori{G}}(u,v)$. If there does not exist a path from $u$ to $v$, we say that $\rho_G(u,v)=\infty$ or $\rho_{\ori{G}}(u,v)=\infty$. We define the \textit{diameter} of $G$ or $\ori{G}$ to be $\diam(G)=\max\{\rho_G(u,v)\mid u,v\in V(G)\}$ and $\diam(\ori{G})=\max\left\{\rho_{\ori{G}}(u,v) \mid u,v\in V(\ori{G})\right\}$ respectively. If $\diam(G)<\infty$, we call $G$ connected. An edge $e\in E(G)$ $(\textrm{or an arc } a\in A(\ori{G}))$ is called a \textit{bridge} if $\diam(G)<\infty$ and $\diam(G\setminus e)=\infty$ $(\textrm{similar for }\ori{G})$. If a graph contains no bridges, we call it \textit{bridgeless}. If $\diam(\ori{G})<\infty$, then we call $\ori{G}$ \textit{strongly connected}.



A classical result, due to Robbins \cite{robbins_theorem_1939}, states that every bridgeless graph has a strongly connected orientation. There may be many such orientations of a graph. A natural next question is what it may mean to find a ``good'' such orientation. Many notions of an objective for optimality of such orientations may be considered. For the purposes of this paper, given a graph $G$, let $\ori{\mathcal{G}}$ represent the set of all strongly connected orientations of $G$. We wish to minimize the oriented diameter of a graph $G$, defined as the following:
$$\ordiam(G)=\min_{\ori{G}\in \ori{\mathcal{G}}}\diam\left(\ori{G}\right).$$

It was shown by Chv\'atal and Thomassen \cite{chvatal_distances_1978} that finding the oriented diameter of a given graph is NP-complete. In the same paper, Chv\'atal and Thomassen found that for the class of bridgeless graphs with diameter $d$, $\ordiam(G)\leq 2d^2+2d$ and constructed bridgeless graphs of diameter $d$ for which every strong orientation admits a diameter of at least $\frac{1}{2}d^2+d$. The upper bound was improved by Babu, Benson, Rajendraprasad and Vaka \cite{babu_improvement_2021} to  $1.373d^2+6.971d-1$.

The paper by Chv\'atal and Thomassen \cite{chvatal_distances_1978} has led to further investigation of such bounds on the oriented diameter given certain graph parameters, including the diameter \cite{fomin_complexity_2004,hutchison_sharp_2007,kwok_oriented_2010}, the radius \cite{chung_strongly_1985}, the domination number \cite{fomin_at-free_2004,laetsch_bounds_2012}, the maximum degree \cite{dankelmann_oriented_2018}, the minimum degree \cite{bau_diameter_2015,czabarka_degree_2019,surmacs_improved_2017}, the number of edges of the graph \cite{cochran_size_2021}, and other graph classes\cite{chen_diameter_2021,gutin_minimizing_1994,gutin_almost_2002,gutin_orientations_2002,hutchison_sharp_2007,koh_orientation_2005,kumar_oriented_2021,lakshmi_optimal_2011,lakshmi_optimal_2007,lakshmi_optimal_2009,plesnik_remarks_1985,soltes_orientations_1986,wang_oriented_2021}. See the survey by Koh and Tay \cite{koh_optimal_2002} for more information on some of these results.

Erd\H{o}s, Pach, Pollack and Tuza \cite{erdhos_radius_1989} proved that the diameter of connected graphs of order $n$ and minimum degree $\delta$ is at most $\frac{3n}{\delta+1}+O(1)$. Bau and Dankelmann \cite{bau_diameter_2015} sought to investigate a similar bound for the oriented diameter and proved that given a bridgeless graph $G$ of order $n$ and minimum degree $\delta$, $\frac{3n}{\delta+1}\leq \ordiam(G)\leq \frac{11n}{\delta+1}.$ The upper bound was improved to $\frac{7n}{\delta+1}$ by Surmacs \cite{surmacs_improved_2017}.

In this paper, we will consider upper bounds on the oriented diameter of a graph considering both the minimum degree $\delta$ and the girth $g$ of a graph. In particular we will prove the following theorem.

\begin{thm}\label{finalTheorem}
	Given $G=(V,E)$, a bridgeless graph of order $n$ and minimum degree $\delta$, there is a polynomial in $\delta$ and $g$, $h(\delta,g)$ of degree $\lfloor \frac{g-1}{2}\rfloor$, for which, given any choice of $\varepsilon>0$, 
	\[
		\ordiam(G)\leq (2g+\varepsilon) \frac{n}{h(\delta,g)}+c.
	\]
\end{thm}

We will also show that in the case of general bridgeless graphs, that $\ordiam(G)\leq (2g+\varepsilon) \frac{n}{\delta+1}+O(1)$. Since bridgeless graphs have a girth $g\geq 3$, we find that if we choose $0<\varepsilon<1$, this gives an improvement on the bound found in the paper by Surmacs \cite{surmacs_improved_2017}.

\section{Preliminaries}
Given a vertex $v\in V(G)$, a natural number $g$, and a path $P$, let $\mathcal{N}(g,v)=\{u \mid \rho_G(u,v)\leq \lfloor \frac{g}{2}\rfloor-1\}$ and $\mathcal{N}(g,v,P)=\{u \mid \rho_{G\setminus E(P)}(u,v)\leq \lfloor \frac{g}{2}\rfloor-1\}.$

\begin{lemma}
	Given a graph $G$ with minimum degree $\delta> 3$, girth $g$, a path $P=p_0p_1\dots p_\ell$, for which $\rho_G(p_i,p_j)=|j-i|$, and a vertex $x\notin V(P),$
	\[
	 	|\mathcal{N}(g,x,P)|\geq 1+\delta+\sum_{i=1}^{\lfloor\frac{g-1}{2}\rfloor-1}\delta(\delta-3)^i.
	 \] 
\end{lemma}

\begin{proof}
	Given a vertex $x\in V(G)$ for which $x\notin V(P)$, $G[\mathcal{N}(g,x)]$ is a tree. If not, there would be a cycle of length less than $g$ in $G$ a contradiction to $g$ being the girth. Since $G[\mathcal{N}(g,x,P)]\subseteq G[\mathcal{N}(g,x)]$, $G[\mathcal{N}(g,x,P)]$ is also a tree.

	We will construct the set $\mathcal{N}(g,x,P)$. Note that $x\in \mathcal{N}(g,x,P)$. Since $x\notin V(P)$, $N(x)\subseteq \mathcal{N}(g,x,P)$ and $|N(x)|\geq \delta$, so $|\{u \mid \rho_{G\setminus E(P)}(v,u)=1\}|\geq \delta$. For each vertex $v_1\in N(x)$, if $v_1\notin V(P)$, then $|N_{G\setminus E(P)}(v_1)|\geq \delta$. If $v_1\in V(P)$, either one or two of the edges incident to $v_1$ are in $E(P)$, so $|N_{G\setminus E(P)}(v_1)|\geq (\delta-2)$. Since $x\in N(v_1)$ we have that $|\{u \mid \rho_{G\setminus E(P)}(v,u)=2\}|\geq \delta(\delta-3).$ Since $\mathcal{N}(g,x,P)$ is a tree, as long as $1\leq i\leq \lfloor\frac{g-1}{2}\rfloor-1$, we can perform a similar analysis to show that $|\{u \mid \rho_{G\setminus E(P)}(v,u)=i+1\}|\geq \delta(\delta-3)^{i}.$ Hence, $|\mathcal{N}(g,x,P)|\geq 1+\delta+\sum_{i=1}^{\lfloor\frac{g-1}{2}\rfloor-1}\delta(\delta-3)^i$.
\end{proof}




 \section{Introduction of Main Lemma}


Let $h(\delta,g)=1+\delta+\sum_{i=1}^{\lfloor\frac{g-1}{2}\rfloor-1}\delta(\delta-3)^i.$ For any $\varepsilon>0$, let $L=\lceil\frac{g-1}{\varepsilon}\rceil$.

\begin{lemma}\label{mainlemma}
	Given a bridgeless graph $G$ with $|G|=n$, girth $g$ and minimum degree $\delta=\delta(G)$, there exists a set of increasing bridgeless subgraphs $H_0\subset H_1\subset H_2\subset \dots H_k \subseteq G$, vertex sets $B_0\subset B_1\subset \dots$ for which $B_i\subseteq V(H_i)$, and a set of forests $F_i$ for which the following hold:
		\begin{enumerate}
			\item \label{Hkspecial} For all $v\in V(G)$, $\rho_G(v,H_k)< L\cdot g$,
			\item \label{forestlarge} for all $i$, $|F_i|\geq h(\delta,g)|B_i|$, and
			\item \label{smallsubgraph} $|H_i|\leq (2g+\varepsilon)|B_i|$.
		\end{enumerate}
\end{lemma}

\begin{proof}
	We will prove by induction on $B_i,F_i,$ and $H_i$. For some $v_0\in V(G)$, let $B_0=\{v_0\}$, $F_0=G[\mathcal{N}(g,v_0)]$, and $H_0=(\{v_0\},\emptyset)$. Certainly property \ref{smallsubgraph} holds. Note that $F_0$ is a tree of order $\sum_{\alpha=0}^{\lfloor \frac{g-1}{2}\rfloor} \delta^\alpha\geq h(\delta,g)$, so property \ref{forestlarge} holds. If property \ref{Hkspecial} holds, we are done.

	Consider $B_i$, $F_i$, $H_i$ for which properties \ref{forestlarge} and \ref{smallsubgraph} hold and property \ref{Hkspecial} does not yet hold. Since property \ref{Hkspecial} does not yet hold, there exists a vertex, $v$, for which $\rho_G(v,H_i)=L\cdot g$. Let $p_0$ be a vertex in $H_i$ for which $\rho_G(v,p_0)=L\cdot g$. Consider a path of shortest length between $p_0$ and $v$, call this path $P=p_0p_1\dots p_{Lg-1}v$ with $v=p_{Lg}$. Let $e_i=p_{i-1}p_i$. Let $H_i'=H_i$. Call $e_j\in E(P)$ \emph{covered} if $e_j$ is not a bridge in $H_i'\cup P$. Let $P_j=p_0\dots p_j$ and $P_j'=p_jp_{j+1}\dots p_{Lg}$. We consider a set of edges $E(P_j)$ to be covered if no edge $e\in E(P_j)$ is a bridge in $H_i'\cup P_j$.   We will build a set of vertices $cov(P)\subseteq V(G)\setminus (V(P)\cup V(H_i)$ which is incident to all the edges used to cover $E(P)$.

	To expand $H_i'$, note that $e_1$ is not covered in $H_i'\cup P$. Since $G$ is bridgeless, there must be a path from $H_i'$ to $P_1'$. Consider a path of length $\rho_{G\setminus E(P)}(H_i',P_1')$, call it $Q$. Note that the two end vertices of $Q$ are the only vertices in $V(Q)$ which can intersect with $V(P).$ Let $p_\beta$ be the end vertex of $Q$ on $P\setminus p_0$. Add $Q$ and $P_\beta$ to $H_i'$. Add the set of interior vertices of $Q$, $V\left(\overline{Q}\right),$ to $cov(P)$, a set of vertices which will eventually be incident to all the edges used to cover $P$. Label the vertices in $cov(P)$ as $q_r$ such that $r=\rho_{G\setminus E(P)}(H_i,q_r)$. Let $B_i'=B_i$. We will now consider an algorithm that will add to $cov(P)$, $B_i'$, and $H_i'$.

	\begin{enumerate}
		\item \label{initial} If there is no longer an edge left uncovered, terminate the algorithm. 
		\item \label{uncovered} If there is an uncovered edge in $P$, consider the edge $e_j$ with the smallest index $j$ that is not yet covered.  Since $G$ is bridgeless, there exists a path from $H_i'$ to $P_j'$ of length $\rho_{G\setminus E(P)}(H_i',P_j')$, call it $R$. Add $V\left(\overline{R}\right)$ to $cov(P)$. Label the vertices $v\in V\left(\overline{R}\right)$ as $q_r$ where $r=|cov(P)|+\rho_{G\setminus E(P)}(H_i',v)$. Add $R$ and $P_j$ to $H_i'$. 
		\item \label{augmentationbegin} If for all pairs of vertices $q_{m_1},q_{m_2}\in cov(P)$ we have $\rho_{G\setminus E(P)}(q_{m_1},H_i)\geq m_1$ and $\rho_{G\setminus E(P)}(q_{m_1},q_{m_2})\geq|m_2-m_1|$, then return to step \ref{initial}. If this was not the case, consider one of the following augmentations.
		\begin{enumerate}
			\item \label{augment1} If $\rho_{G\setminus E(P)}(q_{m_1},H_i)=s< m_1$, remove $\{q_1,\dots q_{m_1-1}\}$ and any edges incident to that vertex set from $H_i'$ and $cov(P)$. Consider a path $S$, which is edge disjoint from $P$ between $q_{m-1}$ and $H_i$ of length $\rho_{G\setminus E(P)}(q_{m_1},H_i)=s$. Add this path to $H_i'$, add the vertices in $V\left(\overline{S}\right)$ to $cov(P)$, and label them $q_\ell$ such that $\ell=\rho_{G\setminus E(P)}(H_i,q_\ell)$. For values from $m_1$ to $t$, where $t$ is the highest current label $r$ for $q_r$ in $cov(P)$, relabel $q_{m_1}\dots q_{t}=q_{s}\dots q_{t-(m_1-s)}$. After relabeling, return to step \ref{augmentationbegin}.
			\begin{figure}
				\scalebox{0.75}{\begin{tikzpicture}
	\begin{pgfonlayer}{nodelayer}
		\node [style=Subgraph] (0) at (-16, 0) {$H_i$};
		\node [style=smallBlack, label={below:$p_1$}] (1) at (-11, 0) {};
		\node [style=smallBlack, label={below:$p_2$}] (2) at (-6, 0) {};
		\node [style=smallBlack, label={below:$p_3$}] (4) at (-3, 0) {};
		\node [style=smallBlack, label={above:$q_1$}] (6) at (-14.75, 1.25) {};
		\node [style=smallBlack, label={above:$q_3$}] (9) at (-12.25, 1.25) {};
		\node [style=smallBlack, label={above:$q_4$}] (10) at (-9.75, 1.25) {};
		\node [style=smallBlack, label={above:$q_6$}] (13) at (-7.25, 1.25) {};
		\node [style=smallBlack, label={left:$q_{7}$}] (18) at (-16, 3) {};
		\node [style=smallBlack, label={right:$q_{11}$}] (19) at (-3, 3) {};
		\node [style=smallBlack, label={above:$q_{10}$}] (20) at (-5.5, 5.25) {};
		\node [style=smallBlack, label={above:$q_{8}$}] (21) at (-13.5, 5.25) {};
		\node [style=smallBlack, label={above:$q_{9}$}] (22) at (-9.5, 5.25) {};
		\node [style=smallBlack, label={above:$q_2$}] (24) at (-13.5, 2.5) {};
		\node [style=smallBlack, label={above:$q_5$}] (25) at (-8.5, 2.5) {};
	\end{pgfonlayer}
	\begin{pgfonlayer}{edgelayer}
		\draw (0) to (1);
		\draw (1) to (2);
		\draw [bend left=45] (18) to (21);
		\draw [bend right=45] (19) to (20);
		\draw (21) to (22);
		\draw (19) to (4);
		\draw (18) to (0);
		\draw (0) to (6);
		\draw (9) to (1);
		\draw (1) to (10);
		\draw (13) to (2);
		\draw (6) to (24);
		\draw (24) to (9);
		\draw (10) to (25);
		\draw (25) to (13);
		\draw (2) to (4);
		\draw (22) to (20);
	\end{pgfonlayer}
\end{tikzpicture}}

				\scalebox{0.75}{\begin{tikzpicture}
	\begin{pgfonlayer}{nodelayer}
		\node [style=Subgraph] (0) at (-6.5, 0) {$H_i$};
		\node [style=smallBlack, label={below:$p_1$}] (1) at (-1.5, 0) {};
		\node [style=smallBlack, label={below:$p_2$}] (2) at (3.5, 0) {};
		\node [style=smallBlack, label={below:$p_3$}] (3) at (6.5, 0) {};
		\node [style=grayedOut] (4) at (-5.25, 1.25) {};
		\node [style=grayedOut] (5) at (-2.75, 1.25) {};
		\node [style=grayedOut] (6) at (-0.25, 1.25) {};
		\node [style=grayedOut] (7) at (2.25, 1.25) {};
		\node [style=smallBlack, label={left:$q_{1}$}] (8) at (-6.5, 3) {};
		\node [style=smallBlack, label={right:$q_{5}$}] (9) at (6.5, 3) {};
		\node [style=smallBlack, label={above:$q_{4}$}] (10) at (4, 5.25) {};
		\node [style=smallBlack, label={above:$q_{2}$}] (11) at (-4, 5.25) {};
		\node [style=smallBlack, label={above:$q_{3}$}] (12) at (0, 5.25) {};
		\node [style=grayedOut] (13) at (-4, 2.5) {};
		\node [style=grayedOut] (14) at (1, 2.5) {};
	\end{pgfonlayer}
	\begin{pgfonlayer}{edgelayer}
		\draw (0) to (1);
		\draw (1) to (2);
		\draw [bend left=45, looseness=1.00] (8) to (11);
		\draw [bend right=45, looseness=1.00] (9) to (10);
		\draw (11) to (12);
		\draw (9) to (3);
		\draw (8) to (0);
		\draw [style=Deleted] (0) to (4);
		\draw [style=Deleted] (5) to (1);
		\draw [style=Deleted] (1) to (6);
		\draw [style=Deleted] (7) to (2);
		\draw [style=Deleted] (4) to (13);
		\draw [style=Deleted] (13) to (5);
		\draw [style=Deleted] (6) to (14);
		\draw [style=Deleted] (14) to (7);
		\draw (2) to (3);
		\draw (12) to (10);
	\end{pgfonlayer}
\end{tikzpicture}}
				\caption{The left graph is an example of subgraph $H'$ where step \ref{augment1} will be executed. The right is $H_i'$ after execution of \ref{augment1}.}
			\end{figure}
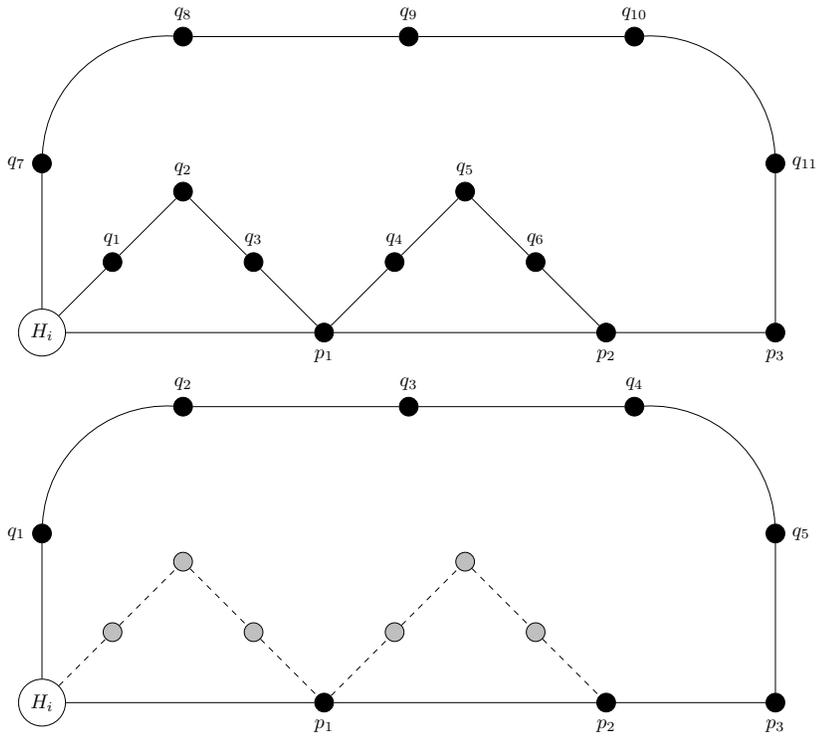
			\item \label{augment2} If $\rho_{G\setminus E(P)}(q_{m_1},q_{m_2})=s<|m_2-m_1|$, without loss of generality, let $m_1<m_2$. Remove the vertices $q_{m_1+1},\dots,q_{m_2-1}$ from $H_i'$ and $cov(P)$. Consider a path $S$, which is edge disjoint from $P$ between $q_{m_1}$ and $q_{m_2}$ of length $\rho_{G\setminus E(P)}(q_{m_1},q_{m_2})=s$. Add this path to $H_i'$, add the vertices in $V\left(\overline{S}\right)$ to $cov(P)$. Label the newly added vertices $q_{m_1+1},\dots,q_{{m_1}+s-1}$ and relabel $q_{m_2+1} \dots q_t=q_{m_1+s}\dots q_{t-((m_2-m_1)-s)}.$ After relabeling, return to step \ref{augmentationbegin}.
			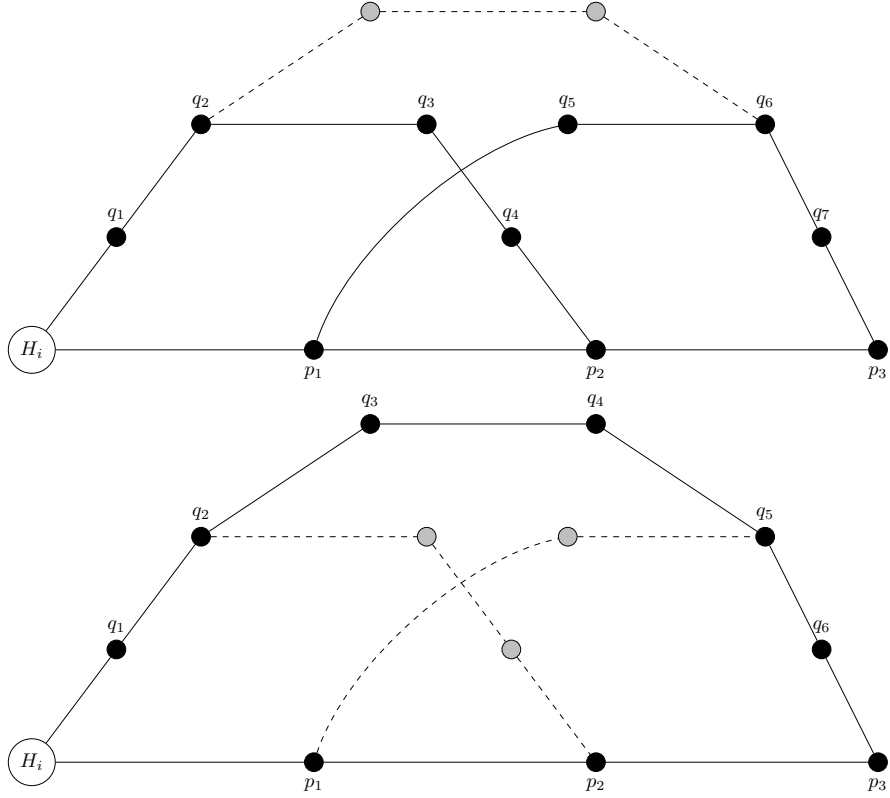
\begin{figure}
				\scalebox{0.75}{\begin{tikzpicture}
	\begin{pgfonlayer}{nodelayer}
		\node [style=Subgraph] (0) at (-16, 0) {$H_i$};
		\node [style=smallBlack, label={below:$p_1$}] (1) at (-11, 0) {};
		\node [style=smallBlack, label={below:$p_2$}] (2) at (-6, 0) {};
		\node [style=smallBlack, label={below:$p_3$}] (3) at (-1, 0) {};
		\node [style=smallBlack, label={above:$q_1$}] (6) at (-14.5, 2) {};
		\node [style=smallBlack, label={above:$q_2$}] (7) at (-13, 4) {};
		\node [style=smallBlack, label={above:$q_3$}] (8) at (-9, 4) {};
		\node [style=smallBlack, label={above:$q_4$}] (9) at (-7.5, 2) {};
		\node [style=smallBlack, label={above:$q_5$}] (10) at (-6.5, 4) {};
		\node [style=smallBlack, label={above:$q_6$}] (11) at (-3, 4) {};
		\node [style=smallBlack, label={above:$q_7$}] (12) at (-2, 2) {};
		\node [style=grayedOut] (13) at (-10, 6) {};
		\node [style=grayedOut] (14) at (-6, 6) {};
	\end{pgfonlayer}
	\begin{pgfonlayer}{edgelayer}
		\draw (0) to (1);
		\draw (1) to (2);
		\draw (2) to (3);
		\draw (0) to (6);
		\draw (6) to (7);
		\draw (7) to (8);
		\draw (8) to (9);
		\draw (10) to (11);
		\draw (11) to (12);
		\draw (9) to (2);
		\draw [bend left, looseness=0.75] (1) to (10);
		\draw (12) to (3);
		\draw [style=Deleted] (7) to (13);
		\draw [style=Deleted] (13) to (14);
		\draw [style=Deleted] (14) to (11);
	\end{pgfonlayer}
\end{tikzpicture}}

				\scalebox{0.75}{\begin{tikzpicture}
	\begin{pgfonlayer}{nodelayer}
		\node [style=Subgraph] (0) at (-16, 0) {$H_i$};
		\node [style=smallBlack, label={below:$p_1$}] (1) at (-11, 0) {};
		\node [style=smallBlack, label={below:$p_2$}] (2) at (-6, 0) {};
		\node [style=smallBlack, label={below:$p_3$}] (3) at (-1, 0) {};
		\node [style=smallBlack, label={above:$q_1$}] (6) at (-14.5, 2) {};
		\node [style=smallBlack, label={above:$q_2$}] (7) at (-13, 4) {};
		\node [style=grayedOut] (8) at (-9, 4) {};
		\node [style=grayedOut] (9) at (-7.5, 2) {};
		\node [style=grayedOut] (10) at (-6.5, 4) {};
		\node [style=smallBlack, label={above:$q_5$}] (11) at (-3, 4) {};
		\node [style=smallBlack, label={above:$q_6$}] (12) at (-2, 2) {};
		\node [style=smallBlack, label={above:$q_3$}] (13) at (-10, 6) {};
		\node [style=smallBlack, label={above:$q_4$}] (14) at (-6, 6) {};
	\end{pgfonlayer}
	\begin{pgfonlayer}{edgelayer}
		\draw (0) to (1);
		\draw (1) to (2);
		\draw (2) to (3);
		\draw (0) to (6);
		\draw (6) to (7);
		\draw [style=Deleted] (7) to (8);
		\draw [style=Deleted] (8) to (9);
		\draw [style=Deleted] (10) to (11);
		\draw (11) to (12);
		\draw [style=Deleted] (9) to (2);
		\draw [style=Deleted, bend left, looseness=0.75] (1) to (10);
		\draw (12) to (3);
		\draw (7) to (13);
		\draw (13) to (14);
		\draw (14) to (11);
	\end{pgfonlayer}
\end{tikzpicture}}
				\caption{The left graph is an example of subgraph $H'$ where step \ref{augment2} will be executed. The right is $H_i'$ after execution of \ref{augment2}.}
			\end{figure}
		\end{enumerate}
	\end{enumerate}


	Any step for which step \ref{augment1} or step \ref{augment2} executes, there was a strict reduction in $|cov(P)|$. On the path $P$, since $\rho(H_i,p_j)=j$, there must be at least $1$ vertex in $cov(P)$, so at some point we must leave step \ref{augmentationbegin} of the algorithm. 
	Any time step \ref{uncovered} executes, there is a strict increase in the number of edges in $P$ that are covered. Since $P$ is finite, at some point the algorithm must return to step \ref{initial} and terminate.

	Let $H_{i+1}=H_i'$, $B_{i+1}=\{B_i\cup q_r|r\equiv 0 \, \mod g\}$, and $F_{i+1}=F_i\bigcup \cup_{b\in B_{i+1}\setminus B_i} \mathcal{N}(g,b,P).$ Now we will show that Properties \ref{forestlarge} and \ref{smallsubgraph} of Lemma \ref{mainlemma} hold. 

To prove Property \ref{smallsubgraph} holds, first remember that $L=\lceil\frac{g-1}{\varepsilon}\rceil\geq \frac{g-1}{\varepsilon}$, hence $g-1\leq L\varepsilon$. We will have two cases: $L \leq |B_{i+1}\setminus B_i|$ and $L > |B_{i+1}\setminus B_i|$. If $L \leq |B_{i+1}\setminus B_i|$, the following holds:
\begin{align}
	\label{inequalities} |H_{i+1}|&\leq |H_{i+1}|+|P|+|cov(P)|\\
	&\leq|H_i|+gL+g |B_{i+1}\setminus B_i| +(g-1)\\
	&\leq |H_i|+g|B_{i+1}\setminus B_i|+g |B_{i+1}\setminus B_i|+L\varepsilon\\
	&\leq |H_i|+g|B_{i+1}\setminus B_i|+g |B_{i+1}\setminus B_i|+ |B_{i+1}\setminus B_i|\varepsilon\\
	&\leq |H_i|+(2g+\varepsilon)|B_{i+1}\setminus B_i|\\
	&\leq (2g+\varepsilon)|B_i| +(2g+\varepsilon)|B_{i+1}\setminus B_i|\\
	\label{inequalitiesend}&\leq (2g+\varepsilon)|B_{i+1}|.
\end{align}

To prove property \ref{forestlarge}, note that for each $b\in B_{i+1}\setminus B_i$, $\rho_{G\setminus E(P)}(b,H_i)\geq g$, otherwise we would have augmented $cov(P)$ in step \ref{augment1} the algorithm, so $\rho_{G\setminus E(P)}(b,B_i)\geq g$. For any pair of vertices $b_1,b_2\in B_{i+1}\setminus B_i$, $\rho_{G\setminus E(P)}(b_1,b_2)\geq g$, otherwise we would have augmented $cov(P)$ in step \ref{augment2} of the algorithm. Hence, $\mathcal{N}(g,b_1,P)\cap \mathcal{N}(g,b_2,P)=\emptyset.$  So, 
\begin{align}
	\label{forestinequalities}\left|F_{i+1}\right|&\geq \left|F_i\right|+\left|\bigcup_{b\in  B_{i+1}\setminus B_i}\mathcal{N}(g,b,P)\right|\\
	&\geq |B_i|\,h(\delta,g)+ |B_{i+1}\setminus B_i|\, h(\delta,g)\\
	\label{forestinequalitiesend}&\geq |B_{i+1}|\,h(\delta,g).
\end{align}

In the case that $L > |B_{i+1}\setminus B_i|$, redefine $B_{i+1}$ to be $B_i\bigcup \cup_{c=1}^L p_{cg}.$ Since $Lg= g|B_{i+1}\setminus B_i| $, the computation above from \ref{inequalities} to \ref{inequalitiesend} holds. See that by definition of $P$, for any $b\in B_{i+1}\setminus B_i$, $\rho_G(b,H_i)\geq g$, so $\mathcal{N}(g,b)\cap V(H_i)=\emptyset$. For any $b_1,b_2\in B_{i+1}\setminus B_i $, $\rho_G(b_1,b_2)\geq g$, hence $\mathcal{N}(g,b_1)\cap \mathcal{N}(g,b_2)=\emptyset$. It follows that 
\begin{align}
	\left|F_{i+1}\right|&\geq \left|F_i\right|+\left|\bigcup_{b\in  B_{i+1}\setminus B_i}\mathcal{N}(g,b)\right|\\
	&\geq |B_i|\,h(\delta,g)+ |B_{i+1}\setminus B_i|\, h(\delta,g)\\
	&\geq |B_{i+1}|\,h(\delta,g).
\end{align}
  Hence, Property \ref{forestlarge} of Lemma \ref{mainlemma} holds in this case.

Since $H_i$ and $B_i$ are increasing subgraphs and vertex sets, and our graph $G$ is a finite graph, eventually \ref{Hkspecial} will hold. When this happens, let $i=k$.
\end{proof}

Now we wish to use Lemma \ref{mainlemma} to create an orientation on a subgraph of $G$ with a small diameter. First, we need to consider the following theorem by Robbins.

\begin{thm}[Robbins\cite{robbins_theorem_1939}]\label{RobbinsTheorem}
	A graph is bridgeless if and only if it admits a strong orientation.
\end{thm}

\begin{lemma}\label{calculationLemma}
	Let $H_k\subseteq G$, $B_k\subseteq V(G)$, and $F_k\subseteq G$, and Properties \ref{Hkspecial}, \ref{forestlarge}, and \ref{smallsubgraph} of Lemma \ref{mainlemma} hold. There exists an orientation of $H_k$, $\ori{H_k}$ for which $\diam\left(\ori{H_k}\right)\leq (2g+\varepsilon) \frac{n}{h(\delta,g)}$.
\end{lemma}

\begin{proof}
	By Property \ref{forestlarge} of Lemma \ref{mainlemma} we have that $h(\delta,g)|B_k|\leq |F_k|\leq n$, so we find that $|B_k|\leq \frac{n}{h(\delta,g)}.$ In conjunction with Property \ref{smallsubgraph} of Lemma \ref{mainlemma}, we find that $|H_k|\leq (2g+\varepsilon)|B_k|\leq (2g+\varepsilon) \frac{n}{h(\delta,g)}.$

	Hence, there exists a bridgeless subraph $H_k$ for which $|H_k|\leq (2g+\varepsilon) \frac{n}{h(\delta,g)}.$ By Theorem \ref{RobbinsTheorem}, there is strong orientation of $H_k$, $\ori{H_k}.$ Note that $\diam\left(\ori{H_k}\right)\leq |H_k| \leq (2g+\varepsilon)\frac{n}{h(\delta,g)}.$
\end{proof}

We now wish to extend our result in Lemma \ref{calculationLemma} for $H_k\subseteq G$ to $G$. To do so, we will need to consider an extension to the following two lemmas, one by Fomin et al. \cite{fomin_at-free_2004} and one by Bau et al. \cite{bau_diameter_2015}

\begin{lemma}[Fomin, Matamala, Prisner and Rapaport \cite{fomin_at-free_2004}]\label{Fomin}
	Let $G$ be a bridgeless graph and $H$ a bridgeless subgraph of $G$ with $\rho_G(v,H)\leq 1$ for all $v\in V(G)$. Given an orientation $\ori{H}$ such that $\diam\left(\ori{H}\right)=d$, then $G$ has an orientation of $d+4$. 
\end{lemma}

\begin{lemma}[Bau and Dankelmann\cite{bau_diameter_2015}]\label{DankelmannBauExtension}
	Let $G$ be a bridgeless graph and $H$ a bridgeless subgraph of $G$ such that $\rho_G(v,H)\leq 2$ for all $v\in V(G)$. Let $\ori{H}$ be a strongly connected orientation of $H$ of diameter $d$. Then there exists a strongly connected orientation of $G$ of diameter at most $d+12$ that extends the orientation of $\ori{H}.$
\end{lemma}

We have that for any $v\in V(G)$, $\rho_G(v,H_k)\leq Lg$. Since $Lg>2$, we will need to extend this lemma as seen below.

\begin{lemma}\label{extension}
	Let $G$ be a bridgeless graph, $H$ a bridgeless subgraph of $G$, and let $s$ be an integer such that $s\geq 2$ and for all $v\in V(G)$, $\rho_G(v,H)\leq s$. Let $\ori{H}$ be a strongly connected orientation of $H$ of diameter $d$. Then there exists a strongly connected orientation of $G$ of diameter at most $d+4\binom{s+1}{2}$ that extends the orientation of $\ori{H}$.
\end{lemma}

\begin{proof}
	Let $H\subseteq G$ be a bridgeless subgraph with an orientation $\ori{H}$ such that $\diam\left(\ori{H}\right)=d$ and $\rho_G(v,H)\leq k$ for all $v\in V(G)$. Let $V_1:=\{v \mid \rho_G(v,H)=1\}.$ Given a vertex $v\in V_1$, label one of its neighbors in $H$ as $x$. Let $\ori{H'}=\ori{H}$, we will continue to augment $\ori{H'}$ throughout the proof. We will call $\ori{H'}$ \emph{extendable} at step $i$ if for any $v\in V(\ori{H'}),$  $\rho_{\ori{H'}}(v,H)+\rho_{\ori{H'}}(H,v)\leq 2i$ and $\rho_{G}(H,v)\leq i$.


	Assume there is a vertex $z$ for which $\rho_G(\ori{H},z)=s$.  First, we will show that there exists a graph $\ori{H'}$ that is extendable at step 1. If there is a vertex $v\in V_1\setminus V\left(\ori{H'}\right)$ for which $\rho_{G\setminus vx}(v,H)=1$, there exists some vertex $y$, $y\neq x$ for which $vy\in E(G)$. Let $\ori{H'}=\ori{H}\cup \ori{xvy}$. Repeat this until there are no longer vertices $v\in V_1\setminus V\left(\ori{H'}\right)$ for which $\rho_{G\setminus vx}(v,H)=1$. Note that for any $v\in V\left(\ori{H'}\right),$  $\rho_{\ori{H'}}(v,H)+\rho_{\ori{H'}}(H,v)\leq 2$ and $\rho_{G}(v,H)\leq 1$, so $\ori{H'}$ is extendable at step 1.

	We will show that for any $1\leq i<2s$, if $\ori{H'}$ is extendable at step $i$, then it is also extendable at step $i+1.$ If there is a vertex $v\in V_1\setminus V\left(\ori{H'}\right)$ for which $\rho_{G\setminus vx}(v,H)=i$, let $Q$ be a path of length $i$ from $v$ to $H$ which does not include $vx$. Consider a vertex $v'\in V(Q)$ for which $v'\in V\left(\ori{H'}\right)$ and $\rho_{G\setminus vx}(v',v)$ is minimized. If $v'\in V(H)$, add $\ori{Q}\cup \ori{xv}$ to $\ori{H'}$. See that for all $v\in V\left(\ori{H'}\right)$, $\rho_{\ori{H'}}(v,H)+\rho_{\ori{H'}}(H,v)\leq 2i$ and $\rho_{G}(v,H)\leq i$, so $\ori{H'}$ is extendable at step $i$. 

	If $v'\notin V(H),$ let $Q'$ be the subpath of $Q$ from $v$ to $v'$. Since $\ori{H'}$ is extendable at step $i$, there exists an integer $j$ for which $|j|<i$, $\rho_{\ori{H'}}(v',H)\leq i-j$, and $\rho_{\ori{H'}}(H,v')\leq i+j$. If $j\geq 0$, add $\oriback{Q'}\cup \ori{vx}$ to $\ori{H'}.$ If $j<0$, add $\ori{Q'}\cup \ori{xv}$ to $\ori{H'}.$ See in each case that for all $v\in V\left(\ori{H'}\right)$, $\rho_{\ori{H'}}(v,H)+\rho_{\ori{H'}}(H,v)\leq 2i$ and $\rho_{G}(v,H)\leq i$.

	Once we have an extendable subgraph $\ori{H'}$ at step $2s$, and have considered all vertices $v\in V_1\setminus V\left(\ori{H'}\right)$ for which $\rho_{G\setminus e}(v,H)\leq 2s$, there are no more vertices $v\in V_1\setminus V\left(\ori{H'}\right)$. If there were a vertex $v\in V_1\setminus V\left(\ori{H'}\right)$ for which $\rho_{G\setminus e}(v,H)>2s$, notice that this would mean there exists a vertex $v'\in V(G)$ for which $\rho_G(v',H)> s$, a contradiction to the assumption of the lemma.

	Since $\ori{H'}$ was extendable at step $2s$, for any $v\in V\left(\ori{H'}\right)$, $\rho_{\ori{H'}}(v,H)\leq 2s$ and $\rho_{\ori{H'}}(H,v)\leq 2s$, so $\ordiam(H')\leq \ordiam(H)+4s$.

\end{proof}

We will now prove Theorem \ref{finalTheorem}.

\begin{proof}
	In Lemma \ref{mainlemma} we showed that there is a bridgeless subgraph $H_k\subseteq G$ such that for any $v\in V(G)$, $\rho_G(v,H_k)\leq Lg$ and \[
		\ordiam(H_k)\leq (2g+\varepsilon) \frac{n}{h(\delta,g)}.
	\]
	By a combination of this and Lemma \ref{extension} with $s=L\cdot g$, we find

	\[
		\ordiam(G)\leq \ordiam(H_k)+\sum_{i=1}^{Lg} 4i \leq (2g+\varepsilon) \frac{n}{h(\delta,g)}+4\binom{Lg+1}{2}.
	\]
\end{proof}

\begin{corollary}
	In Theorem \ref{finalTheorem}, if $g=3$ and $0<\varepsilon <1$, \[
		\ordiam(G)\leq(2g+\varepsilon) \frac{n}{h(\delta,g)}+4\binom{Lg+1}{2}<7 \frac{n}{\delta+1}+O(1).
	\]
\end{corollary}

This is an improvement on the current bound by Surmacs \cite{surmacs_improved_2017}. It is still left as an open question whether this is the smallest possible upper bound in the case without girth. The same question could be asked when including girth as well.

\section{Acknowledgments}

I would like to thank Peter Dankelmann, \'Eva Czabarka, and L\'aszl\'o Sz\'ekely for their mentorship and guidance. I would also like to thank Peter Dankelmann for originally introducting me to this problem and Zhiyu Wang for helping me edit the paper prior to submission.

\bibliographystyle{plain}

\bibliography{References.bib}

\end{document}